\documentclass[10pt]{amsart}
\usepackage{amssymb, amscd, amsmath, amsthm, latexsym, enumerate}

\renewcommand{\geq}{\geqslant}
\renewcommand{\leq}{\leqslant}

\newcommand{\Z}{\mathbb Z}

\newtheorem{theorem}{Theorem}
\newtheorem{lemma}[theorem]{Lemma}
\newtheorem{cor}[theorem]{Corollary}

\theoremstyle{definition}

\newtheorem*{cor*}{Corollary}

\begin{document}
\title{$PD_4$-complexes with $\pi_2$ a projective $\Z[\pi_1]$-module}

\author{Jonathan A. Hillman }
\address{School of Mathematics and Statistics\\
     University of Sydney, NSW 2006\\
      Australia }

\email{jonathanhillman47@gmail.com}

\begin{abstract}
Let $X$ be a $PD_4$-complex and let $\pi=\pi_1(X)$.
If $\pi$ is torsion-free and $\pi_2(X)$ is a finitely generated projective $\Z[\pi]$-module,
then either $\pi$ is free or $\pi$ is $FP$ and $c.d.\pi=4$.
If,  moreover,  $H^3(\pi;\Z[\pi])=0$ then $\pi$ is a free product of $PD_4$-groups and a free group.
\end{abstract}

\keywords{end module, $PD_4$-complex,  projective, second homotopy group}

\maketitle
Let $X$ be a $PD_4$-complex and let $\pi=\pi_1(X)$.  
The study of $PD_4$-complexes with $\pi_2(X)$ a finitely generated free $\Z[\pi]$-module 
began with \cite{Sp03}.
Injectivity of the map $\theta:H^2(X;\Z[\pi])\to{Hom}_{\Z[\pi]}(\Pi,\Z[\pi])$
determined by the Universal Coefficient spectral sequence
is used in an essential way to prove the main result of that paper \cite[Theorem 1.1]{Sp03}.
In \cite[Lemma 2.4]{Sp03} it is asserted that ``[it] is obvious" that
if $\Pi$ is a finitely generated free module then $\theta$ is an isomorphism.
However, we feel that this should either be proven or added to the hypotheses.
We note also that there is an erroneous claim about the possible groups $\pi$,
and this claim is clearly incompatible with $\theta$ being an isomorphism.

We shall give an independent account of this material, 
under the further hypothesis that the end module $H^1(\pi;\Z[\pi])$ 
has projective dimension $\leq1$. 
In \S1 we use the Chiswell-Mayer-Vietoris presentation associated to 
a graph of groups to show that this holds if $\pi$ is torsion-free.
We also present the most frequently used notation and terminology here.
In \S2 we show that if $H^1(\pi;\Z[\pi])$ has finite projective dimension 
and $\Pi=\pi_2(X)$ is a finitely generated projective $\Z[\pi]$-module 
then either $\pi=C_2$ or $\pi$ is free or $c.d.\pi=4$. 
In \S3 we consider the consequences of the condition ``$\Pi$ is projective" alone
for $\pi$ as the fundamental group of a graph of groups.
If we assume also that $\theta$ is an isomorphism then we may
reduce to the case $\Pi=0$ studied in \cite{BBH}.
In \S4 we find that if $\pi$ is torsion-free,
$\Pi$ is projective and $\theta$ is an isomorphism
then $\pi$ is a free product of $PD_4$-groups and a free group,
and the homotopy type of $X$ is determined by the quadratic 2-type together with
the image of the fundamental class in $H_4(\pi;\Z^w)$.
The final section gives several examples, all satisfying the stronger condition $\Pi=0$.

\section{the end module}

The \emph{ end module} of a finitely generated group $\pi$ is the first cohomology group 
$H^1(\pi;\Z[\pi])=Ext^1_{\Z[\pi]}(\Z,\Z[\pi])$.
This has a natural structure as a \emph{right} $\Z[\pi]$-module.
If $\pi$ is $FP_2$ then $\pi\cong\pi\mathcal{G}$, 
where $(\mathcal{G},\Gamma)$ is a finite graph of groups with all vertex groups
either one-ended or finite and all edge groups finite \cite[Theorem VI.6.3]{DD}.
Let $E$ be the set of edges and $V$ the set of vertices in $\Gamma$,
and let $o(e)$ and $t(e)$ be the origin and terminus vertices of the edge $e$.
We shall assume that $G_e$ is a subgroup of $G_{o(e)}$ and 
let $\phi_e$ denote the canonical monomorphism from $G_e$ to $G_{t(e)}$.
We may also assume that $(\mathcal{G},\Gamma)$ is \emph{reduced}, 
meaning that if an edge $e$ has distinct endpoints ($o(e)\not=t(e)$)
then $G_e$ is a proper subgroup of the vertex groups $G_{o(e)}$ and $G_{t(e)}$.
In particular,  $G_v$ is not conjugate to a subgroup of any other vertex group.
Let $T$ be a maximal tree in $\Gamma$.
Then $\pi\mathcal{G}$ is generated by the vertex letters and a stable letter $t_e$, 
one for each edge $e$ \emph{not} in $T$. 
Let $t_e=1$ for $e\in{T}$, for simplicity of notation.

Let $V_f\leq{V}$  the set of vertices with finite vertex group.
If $\pi$ is infinite (so $H^0(\pi;\Z[\pi])=0$) then
there is an associated presentation 
\begin{equation}
\begin{CD}
0\to\oplus_{v\in{V_f}}\Z[G_v\backslash\pi]@>\Psi>>
\oplus_{e\in{E}}\Z[G_e\backslash\pi]\to{H^1(\pi;\Z[\pi])}\to0,
\end{CD}
\end{equation}
arising from a Mayer-Vietoris sequence \cite{Ch76}.
Here $\Psi$ is given by
 \[
\Psi(G_vg)=\sum_{o(e)=v}\sum_{G_eh\leq{G_v}}G_e.hg-
\sum_{t(e)=v}\sum_{G_eh\leq{G_v}}G_e.ht_eg
\]
for all vertices $v\in{V_f}$.
Since $E$ is finite, $H^1(\pi;\Z[\pi])$ is finitely generated.

We use this presentation to simplify the proof of \cite[Lemma 6]{Hi20}.

\begin{lemma}
\label{end module}
\cite[Lemma 2]{Hi20}
Let $\pi=(*_{i=1}^mG_i)*F(r)$, where each factor $G_i$ is $FP_2$ and has one end.
Then
\begin{enumerate}
\item{}if $\pi$ is a non-trivial free group then $H^1(\pi;\Z[\pi])$ has projective dimension $1$;
\item{}if $\pi$ is not free then $H^1(\pi;\Z[\pi])\cong\Z[\pi]^{m+r-1}$.
 \end{enumerate}
\end{lemma}

\begin{proof}
If $\pi\cong{F(r)}$ then $\pi\cong\pi\mathcal{G}$,
where the underlying graph $\Gamma$ has one vertex 
and $r$ edges,  and the vertex and edge groups are all trivial,
while if $\pi\cong(*_{i=1}^mG_i)*F(r)$ then we may assume that $\Gamma$ has $m$ vertices 
and $m+r-1$ edges,  
with the $i$th vertex group being $G_i$ (so $V_f$ is empty),
and all edge groups trivial.
In each case the claim follows from the presentation (1) above.
\end{proof}

If $H^1(\pi;\Z[\pi])$ has finite projective dimension must $\pi$ be torsion-free?
We shall explore further consequences of the presentation (1) in the final section below,
but Lemma \ref{end module} shall suffice for our present purposes.

If $R$ is a right $\Z[\pi]$-module $\overline{R}$ shall denote the left module
with $\pi$-action given by $g.m=w(g)mg^{-1}$, for all $g\in\pi$ and $m\in{M}$,
where $w=w_1(X)$. 
If $L$ is a left $\Z[\pi]$-module let $L^*=\overline{Hom_{\Z[\pi]}(L,\Z[\pi])}$ be 
its conjugate dual.
Let $E^i\Z=\overline{Ext^i_{\Z[\pi]}(\Z,\Z[\pi])}=\overline{H^i(\pi;\Z[\pi])}$ for $i\geq0$.
Then $E^0\Z=0$ if $\pi$ is infinite, while $E^1\Z$ is the conjugate of the end module.
In what follows, the coefficient modules for group homology shall always be left modules.
Thus if $C\leq\pi$ and $M$ is a left $\Z[C]$-module then $H_i(C;M)=Tor_i^{\Z[C]}(\Z,M)$.
We use the $Tor$ notation to emphasize the distinction between left and right modules.

Using Poincar\'e duality and conjugating the exact sequence of \cite [Lemma 3.3]{Hi}
gives an exact sequence of left $\Z[\pi]$-modules
\begin{equation}
\begin{CD}
0\to{E^2\Z}\to\Pi@>ev>>\Pi^*=\overline{Hom_{\mathbb{Z}[\pi]}(\Pi,\mathbb{Z}[\pi])}\to{E^3\Z}\to0,
\end{CD}
\end{equation}
where $e:\Pi\to{\Pi^*}$ is the composition of  a Poincar\'e duality isomorphism 
with the evaluation homomorphism $\theta$.

\section{$\Pi$ finitely generated and projective}

It is easy to see that if $\Pi$ is finitely generated then $\pi$ is $FP_3$.
If $H^1(\pi;\Z[\pi])$ and $\Pi$ are projective modules a stronger result holds.

\begin{theorem}
\label{hom dim}
Let $X$ be a $PD_4$-complex and $\pi=\pi_1(X)$. Then
\begin{enumerate}
\item{}if $\pi$ is finite and non-trivial then $\pi=C_2$ and $X$ is non-orientable;
\item{}if $\pi$ is infinite and $C\leq\pi$ then 
$Tor_i^{\Z[C]}(\Z,E^1\Z)\cong{Tor_{i+4}^{\Z[C]}(\Z,\Z)}$ for $i>1$;
\item{}if $H^1(\pi;\Z[\pi])$ has finite projective dimension and $\Pi=\pi_2(X)$ 
is a finitely generated projective $\Z[\pi]$-module then $\pi$ is $FP$
and either $\pi=C_2$ or $\pi$ is free or $c.d.\pi=4$;
\item{}if $\pi$ has one end then $E^4\Z\cong\Z$.
\end{enumerate}
\end{theorem}

\begin{proof}
The equivariant cellular chain complex for $\widetilde{X}$ is chain homotopy equivalent 
to a complex $C_*$ of finitely generated projective left $\Z[\pi]$-modules, of length 4.
Since $\widetilde{X}$ is 1-connected, $H_0(C_*)\cong\mathbb{Z}$ and $H_1(C_*)=0$,
while $H_2(C_*)\cong\Pi$, by the Hurewicz Theorem, 
and $H_3(C_*)\cong\overline{H^1(\pi;\Z[\pi])}=E^1\Z$,  by Poincar\'e duality.

If $\pi$ is finite and $X$ is orientable then we have an exact sequence
\[
0\to\Z\to{C_4}\to{C_3\oplus\Pi}\to{C_2}\to{C_1}\to{C_0}\to\Z\to0.
\]
Hence $\pi$ has cohomological period dividing 5, and so $\pi=1$,
since a non-trivial finite group must have even cohomological period.
It follows that if $\pi$ is finite and non-trivial then $X$ is non-orientable and $\pi=C_2$.

Suppose now that $\pi$ is infinite. 
Then $H_4(C_*)=0$,  and we have two exact sequences of left $\Z[\pi]$-modules:
\[
0\to{Z_3}\to{C_3\oplus\Pi}\to{C_2}\to{C_1}\to{C_0}\to\Z\to0
\]
and
\[
0\to{C_4}\to{Z_3}\to{H_3(X;\Z[\pi])\cong{E^1\Z}}\to0,
\]
in which $\Pi$ and $C_0,\dots,C_4$ are projective $\Z[\pi]$-modules.
Hence if $C<\pi$ we have 
\[
Tor_i^{\Z[C]}(\Z,E^1\Z)\cong{Tor_i^{\Z[C]}(\Z,Z_3)}\cong{Tor_{i+4}^{\Z[C]}(\Z,\Z)}
%=H_{i+4}(C;\Z)
\quad\mathrm{for}~i>1.
\]
In particular,  if $\pi$ has a non-trivial finite cyclic subgroup $C$ then $E^1\Z$ has exactly one
copy of the augmentation module $\Z$ as a $\Z[C]$-summand,
and the end module has projective dimension $\infty$.

Together these sequences give an exact sequence
\begin{equation}
\begin{CD}
0\to {C_4\oplus{E^1\Z}}\to{C_3\oplus\Pi}\to{C_2}\to{C_1}\to{C_0}\to\Z\to0.
\end{CD}
\end{equation}
If $E^1\Z$ has finite projective dimension,
the augmentation module $\Z$ has finite projective dimension,
since $\Pi$ is projective, and so $c.d.\pi<\infty$.
Hence $\pi$ is torsion-free, 
and so the indecomposable factors of $\pi$ are copies of $\Z$ or have one end.

We may assume $\pi$ is not a free group.
Then $E^1\Z$ is a finitely generated free module, by Lemma \ref{end module},
and so the sequence (3) is a finite projective resolution of length 4.
Hence $\pi$ is $FP$ and $c.d.\pi\leq4$.
Since the dual of this sequence gives the cohomology of $\pi$ with coefficients $\Z[\pi]$,
we see that $E^4\Z=H^4(\pi;\Z[\pi])$ maps onto $H^4(X;\Z[\pi])\cong\Z$.
Therefore either $\pi$ is a free group or $c.d.\pi=4$.
If $\pi$ has one end then $E^4\Z$ maps isomorphically to $H^4(X;\Z[\pi])\cong\Z$.
\end{proof}

In particular, $\pi$ cannot be the fundamental group of an aspherical surface or 3-manifold, 
contrary to an assertion in \cite{Sp03}.
In Theorem 1.1 of \cite{Sp03} it is assumed that $M$ is an orientable 4-manifold
with $\pi=\pi_1(M)$ such that $H_i(\pi;\Z)=0$ for $i=4$ or 5, and
$\pi_2(M)$ is a finitely generated free $\Z[\pi]$-module.
The only groups $\pi$ compatible with these conditions and with the extra condition
here that the end module have finite projective dimension are free groups.

If $\pi$ has one end but is not a $PD_4$-group then either $E^2\Z$ 
is not finitely generated as an abelian group or $E^2\Z=0$ and $E^3\Z$
is not finitely generated as an abelian group \cite[Theorem 3]{Fa75}.

If $\Pi$ is a free module then adding 3-cells to $X$ 
along a basis for $\Pi$ gives a  4-dimensional cell complex $Y$ 
with $\pi_1(Y)\cong\pi$, $\pi_2(Y)=0$,
$H_3(Y;\Z[\pi])\cong{E^1\Z}$ and $H_q(Y;\Z[\pi])=0$ for $q>3$.
(Thus if $\pi$ has one end then $Y\simeq{K(\pi,1)}$.)
When $\Pi$ is projective but not free,  
adding 3- and 4-cells gives such a 4-complex.
Part ($b$) of Theorem \ref{hom dim} then reduces to \cite[Lemma 2.10]{Hi}.
(If $X$ is finite and $\pi$ is of type $FF$ then $\Pi$ is finitely generated and stably free, 
as follows from sequence (3) and Schanuel's Lemma.
In this case $Y$ is finite.)

\section{relaxing the hypothesis on the end module?}

The hypothesis that the end module should have finite projective dimension is unnecessarily strong.
For instance,  if $X=S^1\times\mathbb{RP}^3$ then $\Pi=0$ and $E^1\Z\cong\Z$.
Since $\pi=\Z\oplus{C_2}$ has non-trivial torsion $E^1\Z$ has infinite projective dimension.

In this section we shall explore the consequences of the condition ``$\Pi$ is projective"
for the structure of $\pi$ as the fundamental group of a graph of groups
(with no assumption on the end module).
Most of our results follow from those of \cite[\S7 and \S8]{BBH},
where it is assumed that $\Pi=0$,
but we shall give a self-contained account here, 
which plays off the sequence (1) against Theorem \ref{hom dim}.($b$),
and assumes only that $\Pi$ is projective.

\begin{theorem}
\label{graph of groups}
Let $X$ be an orientable $PD_4$-complex such that $\pi=\pi_1(X)$ is infinite 
and $\Pi=\pi_2(X)$ is projective.
Let $(\mathcal{G},\Gamma)$ be a reduced finite graph of groups with 
all edge groups finite and all vertex groups finite or one-ended, 
and such that $\pi\mathcal{G}\cong\pi$.
Then
\begin{enumerate}
\item{}if $g\in\pi$ has prime power order $p^k$ 
then $g$ is conjugate into some $G_e$;
\item{} if $G_{o(e)}$ and $G_{t(e)} $ each have one end 
then $G_e$ is malnormal in $\pi$;
\item{}if $G_{o(e)}$ is finite while $G_{t(e)}$ has one end then
$[N_{G_{o(e)}}(G_e):G_e]\leq2$;
\item{}if $v$ is an endpoint of just one edge $e$ and $G_v$ is finite 
then $G_e$ normally generates $G_v$,
and $G_v$ is not nilpotent.
\end{enumerate}
\end{theorem}

\begin{proof}
Conjugating the exact sequence (1) gives a presentation
\begin{equation*}
\begin{CD}
0\to\oplus_{v\in{V_f}}\overline{\Z[G_v\backslash\pi]}@>\Upsilon>>
\oplus_{e\in{E}}\overline{\Z[G_e\backslash\pi]}\to{E^1\Z}\to0.
\end{CD}
\end{equation*}
Let $C<\pi$ be a finite subgroup.
Since $C$ acts on the coset spaces $G_v\backslash\pi$ and $G_e\backslash\pi$ by permutations,
the terms $\oplus_{v\in{V_f}}\Z[G_v\backslash\pi]$ and 
$\oplus_{e\in{E}}\Z[G_e\backslash\pi]$ are direct sums of 
right permutation modules $\Z[D\backslash{C}]$, where $D\leq{C}$.
%If $w|_C=1$ then 
Since $\overline{\Z[D\backslash{C}]}\cong\Z[C/D]$ as a left module,
$E^1\Z$ has a presentation by sums of left permutation modules.

If $C$ is cyclic then $H_{2j}(C;E^1\Z)=0$ and $H_{2j+1}(C;E^1\Z)\cong{C}$
for $j>0$,  by part ($b$) of Theorem \ref{hom dim}.
Hence we have exact sequences
\begin{equation}
\begin{CD}
0\to\oplus_{v\in{V_f}}H_{2j+1}(C;\Z[\pi/G_v])\to\oplus_{e\in{E}}H_{2j+1}(C;\Z[\pi/G_e])\to{C}\to0
\end{CD}
\end{equation}
for all $j>0$. 
The sequences are exact at the right since $H_{2j}(C;\Z[C/D])=H_{2j}(D;\Z)=0$ for all $j>0$.

If $g$ has prime power order $p^k$ and $D\leq{C}$ then 
$\Z[C/D]\cong\Z[C/p^jC]$,  for some $j\leq{k}$.
If $w(g)=1$ and $g$ is not conjugate into any edge group then 
$\oplus_{e\in{E}}\Z[\pi/G_e]$ has no $\Z[C]$-summands isomorphic to $\Z$.
Hence the exponent of $\oplus_{e\in{E}}H_{2j+1}(C;\Z[\pi/G_e])$ divides $p^{k-1}$,  for $i>1$.
This contradicts the exactness of sequence (4).

If $G_e\cap{h^{-1}G_eh}\not=1$ for some $h\not\in{G_e}$
then the cosets $1.G_e.$ and $h.G_e$ are each fixed by $G_e\cap{h^{-1}G_eh}$.
If $G_{o(e)}$ and $G_{t(e)} $ each have one end 
then these generate copies of $\Z$ as direct summands of $E^1\Z$.
This contradicts the exactness of sequence (4),
so there can be no such $h$,  and $G_e$ must be malnormal in $\pi$.

Suppose  that $G_{o(e)}$ is finite while $G_{t(e)}$ has one end.
Then $\Upsilon$ restricts to a $\Z[G_e]$-homomorphism from $\Z=\Z[1.G_{o(e)}]$ to 
$\Z[G_{o(e)}/G_e]=\oplus_{G_{o(e)}/G_e}\Z[h.G_e]$,
with torsion-free cokernel.
Each coset of $G_e$ in its normalizer in $G_{o(e)}$ generates a $\Z$ summand of 
$\Z[G_{o(e)}/G_e]$,  with image in $E^1\Z$ a direct summand.
Hence 
\[
[N_{G_{o(e)}}(G_e):G_e]\leq2.
\]

Suppose that $e$ is the only edge with $v$ is an endpoint and $G_v$ is finite.
Let $N_e$ be the normal closure of $G_e$ in $G_v$.
If $N_e\not=G_v$ then there is a prime $p$ dividing $[G_v:N_e]$.
Let $S$ be the Sylow $p$-subgroup of $G_v$.
Then $S\not\leq{N_e}$ (divisibility of orders by $p$),
and so $S$ has is a cyclic group $C\not\leq{N_e}$.
This contradicts ($b$), and so $G_v=N_e$.
In particular, $G_v$ is not nilpotent, since maximal proper subgroups of 
finite nilpotent groups are normal \cite[5.2.41]{Ro}.
\end{proof}

If $X$ is non-orientable then exactness of sequence (4) above breaks down if 
$w=w_1(X)\not=1$,
since $H_{2j}(D;\Z^w)\cong{C_2}$ if $D$ acts nontrivially on $\Z^w$.
We must allow for twisting by $w=w_1(X)$.
We give one such argument, 
corresponding to part (3) of \cite[Theorem 12]{BBH}.

\begin{lemma}
Let $X$ be a $PD_4$-complex such that $\pi=\pi_1(X)\cong\pi\mathcal{G}$,
where $(\mathcal{G},\Gamma)$ be a reduced finite graph of groups with 
all edge groups finite and all vertex groups one-ended,
and such that $\Pi=\pi_2(X)$ is projective.
If  $g\in\pi$ has finite order then $w(g)=1$.
\end{lemma}

\begin{proof}
If no vertex group is finite then the end module is isomorphic to  $\oplus_{e\in{E}}\Z[G_e\backslash\pi]$.
If $g\in\pi$ is orientation reversing then $C=\langle{g}\rangle$ has even order, $2\ell$ say.
The end module is a direct sum of right permutation modules $\Z[D\backslash{C}]$,  
and so $E^1\Z$ is the direct sum of the conjugate left $\Z[\pi]$-modules.
Let $\Z^w$ be the infinite cyclic group with left $\pi$-action determined by $w=w_1(X)$.
If $[C:D]$ is even then $\overline{\Z[D\backslash{C}]}\cong\Z[C/D]$ as a left module,
and $H_i(C;\overline{\Z[D\backslash{C}]})\cong{H_i(D;\Z)}$,  for all $i$.
If $[C:D]$ is odd then $H_i(C;\overline{\Z[D\backslash{C}]})=H_i(D;\Z^w)$,
for all $i$,
as can be seen by comparison of the spectral sequences for $C$ and $D$
as extensions of $\Z^\times$ by their orientation-preserving subgroups.
In either case, $H_i(C;\overline{\Z[D\backslash{C}]})$ has exponent dividing $\ell$.
Hence the summands of $Tor_i^{\Z[C]}(\Z,E^1\Z)$ all have exponent dividing $\ell$,
for all $i$, whereas $H_i(C;\Z)\cong{C}$ for $i$ odd.
Since this contradicts part ($b$) of Theorem \ref{hom dim},
$g$ must be orientation-preserving.
\end{proof}

Theorem 7 of \cite{BBH} asserts that if $g\in\pi$ has finite order $m>1$ and the centralizer 
$C_\pi(g)$ is infinite then $C_\pi(g)$ has two ends and either $w(g)=1$ or $4|m$.
Hence the normalizer $N_\pi(G_e)$ of any edge group is finite or has two ends,
which strengthens part ($d$) of Theorem \ref{graph of groups}.
The argument is based on Crisp's analysis of the actions of finite cyclic groups on trees \cite{Cr00},
and applies with little change to our situation.

The later results of \cite{BBH} depend on the splitting theorem \cite[Theorem B]{BBH},
for $PD_4$-complexes with $\Pi=0$.
At present we have no equivalent result for $PD_4$-complexes with $\Pi$ projective.
(If $\pi=A*B$ is the fundamental group of a $PD_4$-complex with $\Pi$ projective
is the same true for $A$ and $B$?)
However if $\pi$ is virtually free, so all the vertex groups $G_v$ are finite, 
then $E^2\Z=E^3\Z=0$, and so $ev$ is an isomorphism, by the exactness of sequence (2).
In this case we can reduce to the case $\Pi=0$,
and the further conditions on $(\mathcal{G},\Gamma)$ given in \cite[\S7--\S10]{BBH} all hold.

\section{assuming also $E^3\Z=0$}

The homomorphism $ev:\Pi\to\Pi^*$ determines the cohomology intersection pairing $\lambda_X$ 
by the formula $\lambda_X(u,v)=ev(v)(u\cap[X])$, for $u,v\in{H^2(X;\Z[\pi])}$ \cite[\S5]{Hi20}.
(The symbol $\lambda_X$ is used for the equivalent homological formulation in \cite{Sp03}.)
It is not obvious to us that $\Pi$ being free implies that the evaluation map $ev:\Pi\to\Pi^*$ is an isomorphism (as asserted in \cite[Lemma 2.4]{Sp03}),
although we do not know of any counterexample.
(It  \emph{is} obvious that finitely generated free modules are isomorphic to their duals.)
Since this assertion implies that $E^2=E^3=0$, by the exact sequence (2),
this lemma is incompatible with the claim in \cite{Sp03}
that $\Pi$ is free if $\pi$ is the group of an aspherical surface or 3-manifold.

\begin{theorem}
Let $X$ be a $PD_4$-complex and let $\pi=\pi_1(X)$.
If  $\Pi=\pi_2(X)$ is a finitely generated projective $\Z[\pi]$-module and $E^3\Z=0$
then the indecomposable factors of $\pi$ which have one end are finitely presentable $PD_4$-groups.
\end{theorem}

\begin{proof}
If $E^3\Z=0$ then $ev$ is an epimorphism, by the exact sequence (1),
and so $P\cong{P^*}\oplus{E}$, where $E=E^2\Z$.
Dualizing gives $P^*\cong{P^{**}}\oplus{E^*}$,
and so $P^*\cong{P\oplus{E^*}}$, since $P^{**}\cong{P}$.
Let $Q$ be a finitely generated projective complement to $P$,
with $P\oplus{Q}\cong\Z[\pi]^s$.
Then 
\[
\Z[\pi]^s\cong{P\oplus{Q}}\cong{P^*}\oplus{E^*}\oplus{E}\oplus{Q}\cong\Z[\pi]^s\oplus{E^*}\oplus{E},
\]
and so $E^*\oplus{E}=0$, 
since group rings are weakly finite, by a result of Kaplansky.
(See \cite[\S1.5]{Hi} and the references cited there.)
Therefore $ev$ has trivial kernel and so it is an isomorphism.

Since $\Pi$ is  a finitely generated projective $\Z[\pi]$-module and $ev$ is an isomorphism,
there is a 2-connected degree-1 map $f:X\to{P}$ to a $PD_4$-complex $P$
with $\pi_2(P)=0$ \cite[Theorem 5]{Hi20}.

If $\pi$ splits as a free product then $P$ has a parallel decomposition as a connected sum
of $PD_4$-complexes with $\pi_2=0$ \cite[Theorem B]{BBH}.
If $G$ is an indecomposable factor of $\pi$ with one end then the corresponding summand of $X$ is aspherical,  by Poincar\'e duality,  and so $G$ is a $PD_4$-group.
Since $X$ is a $PD_4$-complex, it is finitely dominated, by definition,
and so $\pi$ and its indecomposable factors are finitely presentable.
\end{proof}

The condition ``$E^3\Z=0$"  is satisfied by $PD_n$-groups with $n\not=3$,
and by finite groups and groups with two ends,
and is preserved under free product.
(It also holds for all virtually free groups.)
Thus it is a natural condition to include if each indecomposable factor of $\pi$ has finitely many ends. 

The algebraic notion of Poincar\'e duality group does not require finite presentability.
There is a parallel result for $PD_4$-spaces, 
which are 4-dimensional cell complexes which satisfy Poincar\'e duality of formal dimension 4,
but which are not assumed to be finitely dominated.

\begin{cor}
If $\pi$ is torsion-free then it is a free product of finitely presentable $PD_4$-groups and a free group.
\qed
\end{cor}

Every such group is the fundamental group of a $PD_4$-complex with $\Pi=0$.
Let $G_1,\dots,G_k$ be finitely presentable $PD_4$-groups, 
and let $P_i=K(G_i,1)$ be the corresponding $PD_4$-complexes.
Then the connected sum $Y=\#_{i=1}^kP_i$ is a $PD_4$-complex with
$\pi_1(Y)\cong\pi=*_{i=1}^kG_i$ and $\pi_2(Y)=0$.
We can realize free factors of $\pi$ by connect summing with
copies of $S^3\times{S^1}$ and $S^3\tilde\times{S^1}$, .
If $N$ is a 1-connected $PD_4$-complex  and $\beta=\beta_2(N)$ then
$X=Y\#N$ is a $PD_4$-complex with $\pi_1(X)\cong\pi$ and $\pi_2(X)$ free of rank $\beta$.

Let $\mu=c_{X*}[X]$ be the image of a fundamental class for $X$ under the classifying map 
$c_X:X\to{K(\pi,1)}$. 

\begin{cor}
The homotopy type of $X$ is determined by $\pi$, $w$, $\mu$ and $\lambda_X$.
\end{cor}

\begin{proof}
The complex $P$ is a strongly minimal model for $X$, 
and so the homotopy type of $X$ is determined by $P$ and $\lambda_X$,
by of Corollary 8 of \cite{Hi20}.
Since $\pi_2(P)=0$, the Postnikov 2-stage $P_2(P)$ is just $K(\pi,1)$.
Hence the homotopy type of $P$ is
determined the triple $[\pi,w,\mu]$ \cite{BB08}.
\end{proof}

\section{some examples}

The manifolds $\mathbb{RP}^4$, $\#^r(S^1\times{S^3})$ and $T^4=\mathbb{R}^4/\Z^4$ provide 
simple examples with $\Pi=0$ and $\pi=C_2$, $\pi\cong{F(r)}$ and $c.d.\pi=4$, respectively.
Taking connected sums with 1-connected 4-manifolds gives examples with $\Pi$ a non-zero free module.

Let $j$ be inversion in the abelian group $T^4=\mathbb{R}^4/\Z^4$.
Then $j$ has 16 fixed points (the points of order 2 in the group).
Let $U$ be the complement of a set of small equivariant 4-discs about each of these fixed points.
The involution $j$ acts freely on $U$,
with quotient an orientable 4-manifold $V$ with boundary 16 copies of $\mathbb{RP}^3$.
Let $M$ be the closed 4-manifold obtained by doubling $V$ along its boundary.
Then $\pi=\pi_1(M)\cong\pi\mathcal{G}$, 
where  the underlying graph $\Gamma$ has two vertices and 16 edges, 
the vertex groups are copies of $\Z^4\rtimes_{-1}C_2$ and the edge groups are all $C_2$.
Thus $\pi\cong(\Z^4*\Z^4*F(15))\rtimes{C_2}$.
It is easily seen that $\pi_2(M)=0$, 
and that the edge groups are indeed malnormal in $\pi$.
 
If all the vertex groups are finite then $\pi$ is virtually free, 
and so $E^i\Z=0$ for $i>1$.
Hence $ev$ is an isomorphism, and we may reduce to the case $\Pi=0$,
since we are assuming that $\Pi$ is projective.
In this case the $p$-Sylow subgroups of the vertex groups are cyclic, for $p$ odd.

If $\pi$ has two ends then the maximal finite normal subgroup $F$ must have cohomological period dividing 4.
Conversely, 
if $N$ is a finite group with cohomological period 4 then there is a $PD_3$-complex $Y$ 
with $\pi_1(Y)\cong{N}$, 
and $X=Y\times{S^1}$ is a $PD_4$-complex with $\pi=\pi_1(X)\cong{N}\times\Z$ 
and $\pi_2(X)=0$.
If $N\not=1$ then $H^1(\pi;\Z[\pi])\cong\Z$ has infinite projective dimension as a $\Z[\pi]$-module.

The group $C_5$ acts on $\mathbb{RP}^4$ by cyclically permuting the harmonic coordinates
$[u:v:w:x:y]\mapsto[v:w:x:y:u]$.
There is one fixed point,
and so  $C_5$ acts freely on $\mathbb{RP}^4\#\mathbb{RP}^4$. 
The quotient is indecomposable as a connected sum,
and has fundamental group $D_\infty\times{C_5}$.

The following example shows that the homotopy type of such $PD_4$-complexes
 is not usually determined by the quadratic 2-type alone,
 even when $\Pi=0$.

Let $\Gamma=F(2)/\gamma_3F(2)$ be the Heisenberg group of upper triangular matrices in $SL(3,\Z)$.
Then $\Gamma^{ab}\cong\Z^2$.
Let $A$ be an automorphism of $\Gamma$ whose image $[A]$ in 
$Aut(\Gamma^{ab})\cong{GL(2,\Z)}$ has determinant $-1$ and trace $>0$.
Then $A$ is not conjugate to its inverse.
Every automorphism of $\Gamma$ is orientation preserving, 
and so the semidirect product $G_A=\Gamma\rtimes_A\Z$ is an orientable $PD_4$-group.
Since $G_A$ has abelianization of rank 1, the normal subgroup $A$ is characteristic, 
and since $A$ is not conjugate to its inverse,  
$G_A$ has no orientation reversing automorphism.

This group is the group of an orientable $\mathbb{S}ol^4_1$-manifold $M_A$. 
The connected sums $M_A\#M_A$ and $M_A\#-M_A$ have the same quadratic 2-types
(with trivial intersection pairing!) but are not homotopy equivalent.

In all of these examples $\pi$ is virtually torsion-free, and $ev$ is an isomorphism.
Are there examples without these properties?

%\newpage

\end{document}